\documentclass[reqno,a4paper,11pt]{amsart}

\usepackage{mystyle}
\usepackage{etoolbox}

\usepackage[headings]{fullpage}
\usepackage{parskip}

\usepackage[utf8]{inputenc}
\usepackage[T1]{fontenc}
\usepackage[british]{babel}
\usepackage{microtype,dsfont}

\usepackage[dvipsnames]{xcolor}
\usepackage[pdfencoding=auto,pdfusetitle]{hyperref}
\hypersetup{colorlinks=true, linkcolor=blue, citecolor=teal}

\usepackage{amsfonts,amssymb,bbm,mathrsfs}
\usepackage{amsmath}
\usepackage{mathtools}
\usepackage{centernot,url,booktabs}
\usepackage[noabbrev,capitalize]{cleveref}
\usepackage[shortlabels]{enumitem}
\usepackage{float}

\usepackage{tikz}
\usetikzlibrary{matrix}
\usetikzlibrary{positioning}
\usetikzlibrary{calc}
\usetikzlibrary{arrows}
\tikzset{> =stealth}

\usepackage{tikz-cd}

\newcommand{\addQEDstyle}[2]{\AtBeginEnvironment{#1}{\pushQED{\qed}\renewcommand{\qedsymbol}{#2}}\AtEndEnvironment{#1}{\popQED}}

\theoremstyle{plain}
\newtheorem{theorem}{Theorem}[section]
\newtheorem{lemma}[theorem]{Lemma}
\newtheorem{proposition}[theorem]{Proposition}
\newtheorem{corollary}[theorem]{Corollary}

\theoremstyle{definition}
\newtheorem{definition}[theorem]{Definition}

\addQEDstyle{example}{$\triangle$}

\addQEDstyle{construction}{$\triangle$}

\theoremstyle{remark}

\addQEDstyle{remark}{$\triangle$}

\usetikzlibrary{matrix}
\usetikzlibrary{positioning}
\usetikzlibrary{arrows}
\tikzset{> =stealth}
\usetikzlibrary{arrows.meta}
\tikzset{normalHead/.tip={Triangle[open,angle=60:4pt]},}
\tikzset{normalTail/.tip={Triangle[reversed,open,angle=60:4pt]},}

\newcommand{\cosplitext}[6]{
\tikz[baseline]{
\newdimen{\mylabelwidth}
\newdimen{\skipwidth}
\node[anchor=base] (A) {\hspace*{\dimexpr0.5pt-\pgfkeysvalueof{/pgf/inner xsep}}${#1}$};
\settowidth{\mylabelwidth}{\pgfinterruptpicture {$#2$} \endpgfinterruptpicture}
\pgfmathsetlength{\skipwidth}{max(\mylabelwidth,10pt)}
\node[right] (B) at ([xshift=\skipwidth+12pt]A.east) {${#3}$};
\settowidth{\mylabelwidth}{\pgfinterruptpicture {$#4$} \endpgfinterruptpicture}
\settowidth{\skipwidth}{\pgfinterruptpicture {$#5$} \endpgfinterruptpicture}
\pgfmathsetlength{\skipwidth}{max(\skipwidth,\mylabelwidth,10pt)}
\node[right] (C) at ([xshift=\skipwidth+12pt]B.east) {${#6}$\hspace*{\dimexpr0.5pt-\pgfkeysvalueof{/pgf/inner xsep}}};
\draw[transform canvas={yshift=0.5ex}, normalTail->] (A) to node [above] {${#2}$} (B);
\draw[transform canvas={yshift=0.5ex},-normalHead] (B) to node [above] {${#4}$} (C);
\draw[transform canvas={yshift=-0.5ex},->] (B) to node [below] {${#5}$} (A);
}}

\newcommand{\normalext}[5]{
\tikz[baseline]{
\newdimen{\mylabelwidth}
\newdimen{\skipwidth}
\node[anchor=base] (A) {\hspace*{\dimexpr0.5pt-\pgfkeysvalueof{/pgf/inner xsep}}${#1}$};
\settowidth{\mylabelwidth}{\pgfinterruptpicture {$#2$} \endpgfinterruptpicture}
\pgfmathsetlength{\skipwidth}{max(\mylabelwidth,12pt)}
\node[right] (B) at ([xshift=\skipwidth+10pt]A.east) {${#3}$};
\settowidth{\mylabelwidth}{\pgfinterruptpicture {$#4$} \endpgfinterruptpicture}
\pgfmathsetlength{\skipwidth}{max(\mylabelwidth,10pt)}
\node[right] (C) at ([xshift=\skipwidth+10pt]B.east) {${#5}$\hspace*{\dimexpr0.5pt-\pgfkeysvalueof{/pgf/inner xsep}}};
\draw[normalTail->] (A) to node [above] {${#2}$} (B);
\draw[-normalHead] (B) to node [above] {${#4}$} (C);
}}

\renewcommand{\epsilon}{\varepsilon}
\renewcommand{\phi}{\varphi}

\newcommand{\Gl}{\mathrm{Gl}}

\newcommand{\inv}{^{-1}}

\mathchardef\mhyphen="2D

\title{F-Inverse Monoids as Weakly Schreier Extensions}

\author[P. F. Faul]{Peter F. Faul}
\address{Stellenbosch University, Stellenbosch, South Africa}
\email{peter@faul.io}

\date{18 November 2024}


\begin{document}

\maketitle
\thispagestyle{empty}

\begin{abstract}
It is known that an inverse monoid $M$ is E-unitary if and only if the following diagram is an extension: $\normalext{E(M)}{\subseteq}{M}{q}{M/\sigma}$, where $E(M)$ is the semilattice of idempotents and $q$ is the group quotient map. F-inverse monoids are another fundamental class of inverse semigroup and all F-inverse monoids are E-unitary. Thus given that F-inverse monoids have an associated extension it is natural to ask if these extensions satisfy any special properties. Indeed we show that $M$ is F-inverse if and only if the aforementioned extension is weakly Schreier. This latter result allows us to make use of relaxed factor systems to provide a new characterization of F-inverse monoids. We end by restricting to the Clifford case and find a new characterization of these with much in common with Artin gluings of frames.
\end{abstract}

\section{Introduction}

If one were asked to find a structure that simultaneous generalises a group and a semilattice, a reasonable first guess would be a monoid. However, this is unnecessarily general. There exists an intermediate structure known as an inverse monoid to which both classes of structure belong. 

\begin{definition}
    A monoid $M$ is inverse if for each $x \in M$ there exists a unique $x\inv \in M$ satisfying that $xx\inv x = x$ and $x\inv x x\inv = x\inv$.
\end{definition}

For groups the inverse is the usual group inverse and for semilattices the inverse of an element is just the element itself.

Idempotent elements are especially important in this setting and the element $xx\inv$ is necessarily such. The following is a well-known result.

\begin{theorem}
    Let $M$ be an inverse monoid and let $E(M)$ denote the subset of idempotents. Then $E(M)$ is a subsemilattice of $M$.
\end{theorem}

There is also a natural order structure on an inverse monoid.

\begin{definition}
    Let $M$ be an inverse monoid and let $x,y \in M$. We say $x \le y$ if there exists $e \in E(M)$ such that $x = ey$.
\end{definition}

For more an inverse semigroups see \cite{lawson1998inverse,petrich}.

Of particular importance, at least for the purposes of this paper, is the class of E-unitary semigroups.  

\begin{definition}
    An inverse monoid $M$ is E-unitary if whenever $x \in M$, $e \in E(M)$ and $xe \in E(M)$ then $x \in E(M)$.
\end{definition}

It is always possible to quotient an inverse monoid such that a group results. We call the smallest such congruence $\sigma$. With this concept in hand we can now describe a result in \cite{margolis1987inverse} upon which much of this paper builds upon.

\begin{theorem}
    Let $M$ be an inverse monoid. Then $\normalext{E(M)}{\subseteq}{M}{q}{M/\sigma}$ is an extension if and only if $M$ is E-unitary.
\end{theorem}

Monoid extensions are generally not well-behaved. This was rectified somewhat by Redei after he introduced the notion of Schreier extensions \cite{redei1952verallgemeinerung} and found that this class of extensions were much like extensions of groups. This project has expanded over the years, with many new classes of extensions (of various degrees of well-behavedness) being introduced and studied \cite{leech1975two,grillet1974left,fleischer1981monoid,fulp1971structure}.

In this paper we study F-inverse monoids, a subset of E-unitary monoids.

\begin{definition}
    Let $M$ be an inverse monoid and let $\sigma$ be the minimal group congruence. We call $M$ F-inverse if each $\sigma$-class has a maximal element.
\end{definition}

We show that these are associated to weakly Schreier extensions. An overview of these extensions and their characterization can be found in \cite{faul2022survey}, though they were first studied and characterized in \cite{fleischer1981monoid}.

Given these results we can then apply characterizations of weakly Schreier extensions to F-inverse monoids to provide a new characterization for these objects. This simplifies even further in the Clifford case where we see that they behave somewhat like Artin gluings of frames.

\section{F-inverse monoids}
 In this section we will be concerned with diagrams of inverse monoids of the form \[\normalext{E(M)}{k}{M}{q}{M/\sigma},\] where $k$ the inclusion of the idempotents into $M$ and $q$ is the usual quotient map into $M/\sigma$.

Since all F-inverse monoids are E-unitary, we have that when $M$ is F-inverse, $\normalext{E(M)}{k}{M}{q}{M/\sigma}$ is an extension. 

\begin{definition}
    An extension $\normalext{N}{k}{G}{e}{H}$ is weakly Schreier if there exists a set theoretic splitting $s \colon H \to G$ such that for all $g \in G$ there exists an $n \in N$ such that $g = k(n)se(g)$.
\end{definition}

\begin{theorem}
    Let $M$ be an inverse monoid. Then $\normalext{E(M)}{k}{M}{q}{M/\sigma}$ is a weakly Schreier extension if and only if $M$ is F-inverse.
\end{theorem}

\begin{proof}
    It is clear that the $\sigma$ classes of $M$ are precisely the fibers of $q$.

    The weakly Schreier condition gives an element $s(g)$ in the fibre $e\inv\{g\}$ such that each other element $x$ in the fibre can be generated as some idempotent multiplied by $s(g)$. But this is exactly equivalent to stating that $s(g)$ is the largest element in the fibre.
\end{proof}

\subsection{F-inverse monoids and relaxed semidirect products}

In this section we make use of a characterization of weakly Schreier extensions to explain a resemblance F-inverse monoids have to semidirect products. The following is a well known result and may be found in \cite{petrich}

\begin{definition}\label{res:f-inverse}
    Let $G$ be a group and $Y$ a semilattice. We call $\cdot \colon G \times Y \to Y$ an \emph{almost action} of $G$ on $Y$ if it satisfies the following conditions.
    \begin{enumerate}
        \item $1\cdot y = y$,
        \item $g \cdot (y \wedge z) = g\cdot y \wedge g \cdot z$,
        \item $g \cdot (h \cdot y) = (gh) \cdot y \wedge g \cdot 1$. \label{item:notanaction}
    \end{enumerate}
    Given this data we define $F(Y,G) = \{(y,g) : y \le g \cdot 1\}$ with multiplication given by \[(y,g)(z,h) = (y (g\cdot z), gh).\]
    We call the resulting inverse semigroup an \emph{almost semidirect product}. 
\end{definition}

\begin{theorem}
    Every F-inverse semigroup $S$ is isomorphic to an almost semidirect product. In particular $G$ may be taken to be $S/\sigma$ and $Y$ taken to be $E(S)$.
\end{theorem}

We see that an almost action differs from an action by requiring multiplication by the term $g \cdot 1$ in order to decompose $(gh) \cdot y$ into $g \cdot (h \cdot y)$ and that the resulting almost semidirect product differs from an actual semidirect product by containing only a subset of the product $Y \times G$. 
This latter observation brings to mind $\lambda$-semidirect products of inverse semigroups.



In \cite{faul2021lambda} it was shown that $\lambda$-semidirect products of monoids are weakly Schreier split extensions and in \cite{faul2021characterization} weakly Schreier split extensions were characterized in terms of a \emph{relaxed semidirect product} construction. 

While a relaxed semidirect product construction does not exist for weakly Schreier extensions, there is a characterization analogous to the factor system characterization of group extensions \cite{fleischer1981monoid,faul2022survey}.

\begin{definition}\label{deffac}
Let $H$ and $N$ be monoids, $\sim_-$ an $H$-indexed equivalence relation on $N$, $\cdot \colon H\times N \to N$ and $\chi \colon H \times H \to N$. We call $(\sim_-,\cdot,\chi)$ a factor system when it satisfies the following properties.
\begin{enumerate}
    \item $n_1 \sim_1 n_2$ implies $n_1 = n_2$,
    \item $n_1 \sim_h n_2$ implies $xn_1 \sim_h xn_2$,
    \item $n_1 \sim_{h_1} n_2$ implies $n_1\chi(h_1,h_2) \sim_{h_1h_2} n_2\chi(h_1,h_2)$,
    \item $n_1 \sim_h n_2$ implies $n_1(h\cdot n) \sim_h n_2(h \cdot n)$ for all $n \in N$,
    \item $n_1 \sim_{h_2} n_2$ implies $(h_1 \cdot n_1)\chi(h_1,h_2) \sim_{h_1h_2} (h_1 \cdot n_2)\chi(h_1,h_2)$ for all $h_1 \in H$,
    \item $(h \cdot n_1n_2) \sim_h (h \cdot n_1)\cdot(h\cdot n_2)$,
    \item $\chi(h_1,h_2)(h_1h_2 \cdot n) \sim_{h_1h_2} (h_1 \cdot (h_2 \cdot n))\chi(h_1,h_2)$,
    \item $(h \cdot 1) \sim_h 1$,
    \item $(1 \cdot n) \sim_1 n$,
    \item $\chi(1,h) \sim_h 1 \sim_h \chi(h,1)$,
    \item $\chi(x,y)\chi(xy,z) \sim_{xyz} (x \cdot \chi(y,z))\chi(x,yz)$.
\end{enumerate}
\end{definition}

This is just like the usual characterization of a group extension, except that an $H$-indexed equivalence relation has been introduced and the usual axioms of an action and factor set are required now to only hold up to $H$-equivalence.

\begin{definition}
Given a factor system $(\sim_-,\cdot,\chi)$ of $H$ on $N$ we can define a monoid $N \ltimes_\chi H$ with underlying set $\bigsqcup_{h \in H}N/{\sim}_h$ and multiplication \[([n],h)([n'],h') = ([n (h\cdot n') \chi(h,h')],hh').\] 
\end{definition}

If $(\sim_-,\cdot,\chi)$ is the relaxed factor system associated to some weakly Schreier extension, the above monoid is isomorphic to the center of that extension.

Now we know that F-inverse monoids correspond to weakly Schreier extensions and so have a characterization in terms of these relaxed factor sets. Of interest is that this characterization can be related to the almost actions and almost semidirect product introduced above. 

\begin{proposition}
    Consider an almost action $\cdot$ of a group $G$ on a semilattice of $Y$. Let $y \sim_g z$ if and only if $y \wedge (g \cdot 1) = z \wedge (g \cdot 1)$, and $\chi(g,h) = g \cdot 1$. Then $(\sim_-,\cdot,\chi)$ is a factor system
\end{proposition}

\begin{proof}
    Let us go through each of the 11 conditions individually.
    \begin{enumerate}
        \item Assume $y_1 \sim_1 y_2$ and note that $y_1 = y_1 \wedge (1 \cdot 1) = y_2 \wedge (1 \cdot 1) = y_2$.
        \item Assume $y_1 \sim_g y_2$ and consider the following calculation: $x \wedge y_1 \wedge (g \cdot 1) = x \wedge y_2 \wedge (g \cdot 1)$.
        \item Assume $y_1 \sim_{g_1} y_2$ and consider $y_1 \wedge \chi(g_1,g_2) \wedge (g_1g_2 \cdot 1) = y_1 \wedge (g_1 \cdot 1) \wedge (g_1g_2 \cdot 1) = y_2 \wedge (g_1 \cdot 1) \wedge (g_1g_2 \cdot 1)$.
        \item Assume $y_1 \sim_h y_2$ and let $y \in N$, then $y_1 \wedge (g \cdot y) \wedge (g \cdot 1) = y_2 \wedge (g \cdot y) \wedge (g \cdot 1)$.
        \item Assume $y_1 \sim_{g_2} y_2$ and let $g_1 \in G$. We have $(g_1 \cdot y_1)(g_1 \cdot 1) \wedge (g_1g_2 \cdot 1) = (g_1 \cdot y_1) \wedge g_1(g_2 \cdot 1) = g_1 \cdot (y_1 \wedge (g_2 \cdot 1)) = g_1 \cdot (y_2 \wedge (h_2 \cdot 1)) = g_1 \cdot y_2)(g_1 \cdot 1) \wedge (g_1g_2 \cdot 1)$.
        \item This condition follows automatically from the fact that $\cdot$ is an almost action.
        \item Consider the following calculation: $(g_1 \cdot 1) \wedge (g_1g_2 \cdot y) \wedge (g_1g_2 \cdot 1) = g_1 \cdot (g_2 \cdot n) \wedge (g_1 \cdot 1) \wedge (g_1g_2 \cdot 1)$.
        \item This follows trivially.
        \item This also follows trivially.
        \item Consider $(1 \cdot 1) \wedge (h \cdot 1) = 1 \wedge (h \cdot 1) = (h \cdot 1) \wedge (h \cdot 1)$
        \item Finally, $(x \cdot 1) \wedge (xy \cdot 1) = x \cdot (y \cdot 1) \wedge (x \cdot 1)$.
    \end{enumerate}
\end{proof}

\begin{theorem}
    Let $\cdot$ be an almost action of a group $G$ on a semilattice $Y$ and let $(\sim_-,\cdot,\chi)$ be the associated factor system. Then $F(Y,G) \simeq Y \ltimes_\chi G$.
\end{theorem}

\begin{proof}
    Let $\phi \colon F(Y,G) \to Y \ltimes_\chi G$ send $(y,g)$ to $([y]_g,g)$.

    To see that $\phi$ is a homomorphism consider the following calculation.

    \begin{align*}
        \phi(y,g)\phi(z,h)  &= ([y]_g,g)([z]_h,h) \\
                            &= ([y \wedge (g \cdot z)]_gh,gh) \\
                            &= \phi(y \wedge (g \cdot z),gh) \\
                            &= \phi((y,g)(z,h)).
    \end{align*}    

    To show that $\phi$ is an isomorphism consider the candidate inverse $\psi\colon Y \ltimes_\chi G \to F(Y,G)$ which sends $([y]_g,g)$ to $(y \wedge (g \cdot 1),g)$. Since $y \sim_g z$ if and only if $y \wedge (g \cdot 1) = z \wedge (g \cdot 1)$ we have that $\psi$ is well-defined.

    First we consider $\phi\psi([y]_g,g) = \phi(y \wedge (g \cdot 1), g) = ([y \wedge (g \cdot 1)]_g,g) = ([y]_g,g)$. 

    Finally consider $\psi\phi(y,g) = \psi([y]_g,g) = (y \wedge (g \cdot 1),g)$. Recall that $(y,g) \in F(Y,G)$ only if $y \le g \cdot 1$ and so $\psi\phi(y,g) = (y,g)$.
\end{proof}

This explains the apparent similarities between $F(Y,G)$ and a semidirect product and provides a new way of thinking about F-inverse semigroups --- as pairs of elements $([y]_g,g)$ and with multiplication $([y],g)([z],h) = ([y \wedge (g \cdot z) \wedge (g \cdot 1)],gh)$.

\subsection{Clifford semigroups and Artin gluings}

The extensions associated to Clifford semigroups have some unique aspects to them. For instance, if $M$ is a Clifford semigroup, the kernel map $k$ has a splitting $\ell\colon M \to E(M)$, such that $\ell(m) = mm\inv$. In general such a mapping will not be a homomorphism.

This makes it into something that may be termed a cosplit extension $\cosplitext{E(M)}{k}{M}{q}{\ell}{M/\sigma}$, where the kernel map has a retraction $\ell$. We make little use of this observation in the remained of this paper, though it clearly warrants further investigation. Instead we focus on a simpler characterization of Clifford F-inverse monoids. 

Artin gluings are a topological construction which allow for two spaces to be embedded into a larger space with one open and the other its closed complement \cite{wraith1974artin}. This construction has an analogue for frames, the algebraic structure corresponding to a lattice of open sets. Frames may be viewed as particular kinds of ordered monoids and it is in this context that we can generalise the construction.

\begin{definition}
    Let $H$ and $N$ be commutative, ordered monoids and let $f \colon H \to N$ be a monoid homomorphism. Then we can construct the Artin gluing $\Gl(f) = \{(n,h) \in N \times H : n \le f(h)\}$ with pointwise multiplication. 
\end{definition}

Artin gluings naturally form weakly Schreier split extensions \cite{faul2021lambda}. Our extensions are not split and so we introduce the following modification to the definition.

\begin{definition}
    Let $G$ be a group and $Y$ a semilattice and let $f\colon G \to Y$ satisfy that $f(gh) \wedge f(g) = f(g) \wedge f(h)$. We define the gluing of $f$ as \[\Gl(f) = \{(y,g) \in Y \times G : y \le f(g)\},\]
    with multiplication $(y,g)(z,h) = (y \wedge z,gh)$.
\end{definition}

For the traditional Artin gluing one asks that $f(gh) = f(g) \wedge f(h)$. Our modification accords with the fact that the resulting extension will be only weakly Schreier and not weakly Schreier split.

\begin{proposition}
    Let $G$ be a group and $Y$ a semilattice and let $f\colon G \to Y$ satisfy that $f(gh) \wedge f(g) = f(g) \wedge f(h)$. Then $\Gl(f)$ is an F-inverse Clifford semigroup.
\end{proposition}


\begin{proof}
    The idempotent elements are those of the form $(y,1)$ and it is evident from the definition of multiplication that these are central.
    We demonstrate that $\Gl(f)$ is F-inverse by forming the associated weakly Schreier extension $\normalext{Y}{k}{\Gl(f)}{e}{G}$ with $k(y) = (y,1)$ and $e(y,g) = g$. The map $k$ is clearly the kernel of $e$ and so it suffices to demonstrate that the weakly Schreier condition holds. Let $s(g) = (f(g),g)$ and note that $(y,g) = (y,1)(f(g),g)$. 
\end{proof}

\begin{corollary}
    Let $M$ be an F-inverse Clifford semigroup with $M/\sigma$ abelian. Then $M$ is commutative.
\end{corollary}

In fact, every F-inverse Clifford semigroup can be represented in this way. 

\begin{lemma}\label{res:glue}
    Let $M$ be an F-inverse, Clifford monoid and consider the function $f\colon M/\sigma \to E(M)$ satisfying that $f(g) = s(g)s(g)\inv$ where $s(g)$ is the largest element in the equivalence class $g$. Then $f(gh) \wedge f(g) = f(g) \wedge f(h)$.
\end{lemma}

\begin{proof}

    We prove this result in two steps. Note that $s(g)\inv = s(g\inv)$.

    \begin{align*}
        f(g) \wedge f(h)    &= s(g)s(g)\inv s(h)s(h)\inv \\
                            &= s(g)s(h)s(h)\inv s(g)\inv \\
                            &= s(g)s(h)s(h)\inv s(g)\inv s(g)s(g)\inv \\
                            &\le s(gh)s(h\inv g\inv) s(g)s(g)\inv \\
                            &= s(gh)s(gh)\inv s(g)s(g)\inv \\
                            &= f(gh) \wedge f(g).
    \end{align*}

    Note that line 2 follows because idempotents are central by assumption. Line 4 follows because $s(g)s(h)$ belongs to the $gh$ fibre and so compares less than $s(gh)$ (and similarly for $s(h)\inv s(g) \inv$).

    Next consider the following calculation.

    \begin{align*}
        f(gh) \wedge f(g)   &= s(gh)s(gh)\inv s(g)s(g)\inv \\
                            &= s(g)s(g)\inv s(gh)s(h\inv g\inv) s(g)s(g)\inv \\
                            &\le s(g)s(h)s(h)\inv s(g)\inv \\
                            &= f(g) \wedge f(h).
    \end{align*}

    Together this proves the result.
    
\end{proof}

\begin{theorem}
    Let $M$ be an F-inverse, Clifford monoid and define $f\colon M/\sigma \to E(M)$ as in \cref{res:glue}. Then $\Gl(f) \simeq M$.
\end{theorem}

\begin{proof}
    Consider the associated cosplit extension $\cosplitext{E(M)}{k}{\Gl(f)}{q}{\ell}{M/\sigma}$.
    Let $s\colon M/\sigma \to M$, select the largest element in each fibre. Define $\phi\colon M \to \Gl(F)$, $\phi(x) = (xx\inv,[x]) = (\ell(x),q(x))$ and its candidate inverse $\psi\colon \Gl(f) \to M$, $\psi(y,g) = k(y)s(g)$.

    First we show that $\phi$ is well defined. By definition, $x \le sq(x)$ and so $\ell(x) = xx\inv \le sq(x)sq(x)\inv$ as required. Next we show that it is a monoid homomorphism. Consider the following calculation.

    \begin{align*}
        \phi(x)\phi(y)  &= (xx\inv,[x])(yy\inv,[y]) \\
                        &= (xx\inv yy\inv, [xy]) \\
                        &= (xyy\inv x\inv, [xy])\\
                        &= (xy)(xy)\inv,[xy]\\
                        &= \phi(xy).
    \end{align*}

    It remains to show that $\psi$ is the inverse of $\phi$. 

    \begin{align*}
        \phi(\psi(y,g)) &= \phi(k(y)s(g)) \\
                        &= (k(y)s(g)s(g)\inv k(y)\inv, [k(y)s(g)]) \\
                        &= (k(y)k(y)\inv s(g)s(g)\inv,g) \\
                        &= k(y)s(g)s(g)\inv,g) \\
                        &= (y,g).
    \end{align*}    

    Here the last line follows because $k$ is simply the inclusion and $(y,g) \in \Gl(f)$ implies $y \le s(g)s(g)\inv$.

    Finally consider the following equation.

    \begin{align*}
        \psi(\phi(x))   &= \psi(xx\inv,[x]) \\
                        &= k(xx\inv)s([x]) \\
                        &= x.
    \end{align*}    

    Here the last line follows because when $x \le y$ we know $x = xx\inv y$.
\end{proof}

\bibliographystyle{abbrv}
\bibliography{bibliography}
\end{document}